\documentclass[a4paper,12pt]{article}

\usepackage{palatino}
\usepackage{mathpazo}

\usepackage{mathrsfs}
\usepackage{latexsym,amssymb,amsmath,amsthm,url}
\usepackage{MnSymbol}
\usepackage{verbatim}
\usepackage[small]{titlesec}

\usepackage{enumitem}

\theoremstyle{plain} 
\newtheorem*{mainthm}{Main Theorem}
\newtheorem{lemma}{Lemma}[section]
\newtheorem{thm}[lemma]{Theorem}

\theoremstyle{definition}

\newtheorem{question}[lemma]{Question}

\theoremstyle{remark}

\newcommand{\N}{\mathbb{N}}
\newcommand{\Q}{\mathbb{Q}}
\newcommand{\st}{\::\:}
\newcommand{\pst}{\:|\:}

\DeclareMathOperator{\lcm}{lcm}

\title{On disjoint unions of finitely many copies of the free monogenic semigroup}

\author{N. Abu-Ghazalh\thanks{The first author is financially supported by The Ministry of Higher Education in Saudi Arabia (Princess Nora Bint Abdul Rahman University in Riyadh, Ref number RUG0003).},  N. Ru\v{s}kuc} 

\begin{document}

\maketitle

\begin{abstract}
Every semigroup which is a finite disjoint union of copies of the free monogenic semigroup (natural numbers under addition) is finitely presented and residually finite.

\bigskip

\end{abstract}

\section{Introduction} 
\label{sec1}

Unlike the 'classical' algebraic structures, such as groups and rings, it is well known that a semigroup may decompose into a disjoint union of subsemigroups.
Indeed many structural theories of semigroups have such decompositions at their core.
For example: 

\begin{itemize}[label=\textbullet, leftmargin=*]
\item
Every completely simple semigroup is isomorphic to a Rees matrix semigroup over a group $G$, and is thus a disjoint union of copies of $G$; see \cite[Theorem 3.3.1]{howie95}.
\item
Every Clifford semigroup is isomorphic to a strong semilattice of groups, and is thus a disjoint union of its maximal subgroups; see \cite[Theorem 4.2.1]{howie95}.
\item
Every commutative semigroup is a disjoint union (indeed a semilattice) of archimedean commutative semigroups; see \cite[Theorem 4.2.2]{grillet95}.
\end{itemize}

It is therefore natural to ask how properties of a semigroup $S$ which can be decomposed into 
a disjoint union of subsemigroups $S=T_1\cupdot \dots \cupdot T_n$ depend on properties of the $T_i$.
For instance, it is obvious that if all $T_i$ are finitely generated then so is $S$.
Araujo et al. \cite{araujo01} discuss finite presentability in this context, and show that there exists a non-finitely presented semigroup which is a disjoint union of two finitely presented subsemigroups.
On the other hand, it can be shown that in many special instances finite presentability of the $T_i$ implies
finite presentability of $S$. For example, this is the case when all $T_i$ are groups (i.e. when $S$ is a completely regular semigroup with finitely many idempotents; see \cite[Theorem 4.1.1]{howie95}); this follows from \cite[Theorem 4.1]{ruskuc99}.
Further such instances are discussed in \cite{araujo01}.

Turning to the finiteness condition of residual finiteness, we have a similar landscape.
It is easy to construct a non-residually finite semigroup which is a disjoint union of two residually finite subsemigroups. One such example, consisting of a free group and a zero semigroup, can be found in
\cite[Example 5.6]{gray12}.

On the other hand, it follows from Golubov \cite{golubov75} that if all $T_i$ are residually finite groups then $S$ is residually finite as well.

In this paper we consider semigroups which are disjoint unions of finitely many copies of the free monogenic semigroup. 
Throughout the paper we will denote this semigroup by $N$; hence $N$ is a multiplicative isomorphic copy of the additive semigroup $\N$ of natural numbers.
We show that even though there is no general structural theory for such semigroups, which would yield positive results of the above type `for free', they nonetheless display the same behaviour as unions of groups:

\begin{mainthm}
Every semigroup which is a finite disjoint union of copies of the free monogenic semigroup is finitely presented and residually finite.
\end{mainthm} 

We remark that it immediately follows from general theory that all such semigroups have decidable word problem and are hopfian. In fact, our proof of finite presentability provides an explicit solution to the word problem.

Subsemigroups of the free monogenic semigroups, and in particular the so called \emph{numerical semigroups}
(subsemigroups of finite complement) have been subject to extensive investigation over the years; see \cite{rosales09} for a comprehensive introduction. In a sense in this paper we take a complementary view-point: instead of looking at subsemigroups of $N$, we investigate semigroups which are `composed' of finitely many copies of $N$.

The paper is organised as follows. In Section \ref{sec2} we undertake an analysis of multiplication
in a semigroup under investigation, exhibiting certain strong regularities which are all described in terms of arithmetic progressions. These results are utilised to prove the finite presentability part of the Main Theorem in Section \ref{sec3} and residual finiteness in Section \ref{sec4}. Finally in Section \ref{sec5} we pose some questions which in our opinion point the way for interesting future investigations.

\section{Preliminaries: multiplication and arithmetic progressions}
\label{sec2}

Let $S$ be a semigroup which is a disjoint union of $n$ copies of the free monogenic semigroup:
\[
S=\bigcupdot_{a\in A} N_a,
\]
where $A$ is a finite set and $N_a=\langle a\rangle$ for $a\in A$.
In this section we gather some background facts about $S$. The common feature is that they all elucidate a strong regularity with which elements of $S$ multiply. We begin with two preliminary lemmas.

\begin{lemma}
\label{lemma1}
 Let $a\in A$ and $q\in\N$ be fixed. There can be only finitely many elements $x\in S$ such that 
$a^p x=a^{p+q}$ for some $p\in\N$.
\end{lemma}
 
\begin{proof}
 Suppose to the contrary that there are infinitely many such $x$. 
Two of these elements must belong to the same block $N_c$. Suppose these elements are $c^r$ and $c^s$ for $r\neq s$, and suppose we have
$a^{p_1}c^r=a^{p_1+q}$ and $a^{p_2}c^s=a^{p_2+q}$.
Note that these equalities imply $a^p c^r=a^{p+q}$ for all $p\geq p_1$, and 
$a^p c^s=a^{p+q}$ for all $p\geq p_2$.
Let $p=\max(p_1,p_2)$, and evaluate the element $a^p c^{rs}$ twice:
\[
 a^p c^{rs} = a^p(c^r)^s =  a^p\underbrace{c^r\dots c^r}_{s}=
a^{p+q}\underbrace{c^r \dots c^r}_{s-1}
=\dots = a^{p+sq},
\]
and, similarly,
\[
a^p c^{rs} =  a^p(c^s)^r=a^{p+rq}.
\]
 But from $r\neq s$ it follows that $a^{p+sq}\neq a^{p+rq}$, a contradiction.
\end{proof}

\begin{lemma}
 \label{lemma2}
 If $a^pb^q=a^r$ holds in $S$ for some $a,b\in A$ and $p,q,r\in \N$ then $p\leq r$.
 \end{lemma}
 
\begin{proof}
 Suppose to the contrary that $r=p-s<p$. Note that for every $t\geq p$ we have
\[
a^t\cdot a^sb^q = a^{t+s}b^q = a^{t+s-p} a^pb^q=a^{t+s-p}a^r = a^{t+s-p+p-s} = a^t.
\]
 Hence, for every $u \geq1$ we have 
\[
a^t(a^sb^q)^u a^s=a^t a^s= a^{t+s}.
\]
 By Lemma \ref{lemma1} we must have 
\[
(a^sb^q)^u a^s=(a^s b^q)^v a^s
\]
 for some distinct $u,v\in\N$. 
Post-multiplying by  $b^q$ we obtain
\[
(a^sb^q)^{u+1}=(a^sb^q)^{v+1}.
\] 
This means that the element $a^s b^q\in S$ has finite order, a contradiction.
\end{proof}

The next result shows that multiplication by $x\in S$ cannot 'reverse' the order of elements
from the copies of $N$:

\begin{lemma}
\label{lemma3}
If $a,b\in A$ and $x\in S$ are such that
\[
a^px=b^r,\ a^{p+q}x=b^s
\]
for some $p,q,r,s\in \N$, then $r\leq s$.
\end{lemma}

\begin{proof}
The assertion follows from
\[
a^qb^r=a^qa^px=a^{p+q}x=b^s,
\]
and (the dual of) Lemma \ref{lemma2}.
\end{proof}

The next lemma is absolutely pivotal for proofs of both finite presentability and residual finiteness:

\begin{lemma}
\label{lemma4}
If 
\begin{equation}
\label{eq2}
a^px=b^r,\ a^{p+q}x=b^{r+s}
\end{equation}
for some $a,b\in A$, $x\in S$, $p,q,r\in\N$, $s\in\N_0$, then
\[
a^{p+qt}x=b^{r+st}
\]
for all $t\in\N_0$.
\end{lemma}

\begin{proof}
   First note that from \eqref{eq2} we have 
   \begin{equation}
   \label{eq3}
   b^{r+s}=a^{p+q}x= a^{q}a^px=a^{q}b^r.
   \end{equation}
   We now prove the lemma by induction on $t$.
   For $t=0$ we get the first relation in \eqref{eq2}.
  Assume the statement holds for some $t$. Then, by induction and \eqref{eq3},
\[
a^{p+q(t+1)}x= a^{q}a^{p+qt}x=a^{q}b^{r+st}= a^{q}b^r b^{st}
=b^{r+s}b^{st}= b^{r+s(t+1)},
\]
proving the lemma.
\end{proof}

Motivated by Lemma \ref{lemma4} we introduce the sets
\[
T(a,x,b)=\{ y\in N_a\st yx\in N_b\}\ (a,b\in A,\ x\in S).
\]
The following is immediate:

\begin{lemma}
\label{lemma4a}
For any $a\in A$ and $x\in S$ we have
\[
N_a=\bigcupdot_{b\in A} T(a,x,b).
\]
\end{lemma}

By Lemmas \ref{lemma3}, \ref{lemma4}, if a set $T(a,x,b)$ 
contains more than one element, then it contains an arithmetic progression, and hence is infinite.
In fact, if $T(a,x,b)$ is infinite then it actually stabilises into an arithmetic progression:

\begin{lemma}
\label{lemma5}
If $T=T(a,x,b)$ is infinite then there exist sets $F=F(a,x,b)$, $P=P(a,x,b)$ such that the following hold:
\begin{enumerate}[label=\textup{(\roman*)}, leftmargin=*,widest=iii]
\item
\label{T0}
$T=F\cupdot P$;
\item
\label{T1}
$P=\{ a^{p+qt}\st t\in\N_0\}$ for some $p=p(a,x,b),q=q(a,x,b)\in\N$ and $a^{p-q}\not\in T$;
\item
\label{T2}
$F\subseteq \{a,\dots,a^{p-1}\}$ is a finite set.
\end{enumerate}
\end{lemma}

\begin{proof}
Let $q\in\N$ be the smallest number such that $a^p,a^{p+q}\in T$ for some $p\in\N$.
Furthermore, let $p$ be the smallest such; in particular $a^{p-q}\not\in T$.
Let $P=\{ a^{p+qt}\st t\in\N_0\}$.
By Lemmas \ref{lemma3}, \ref{lemma4} we have $P\subseteq T$,
and by minimality of $q$ we have $a^{p+tq+r}\not\in T$ for any $t\in\N_0$ and any $r\in\{1,\dots,q-1\}$.
Hence $F=T\setminus P\subseteq\{ a,\dots,a^{p-1}\}$, and the lemma is proved.
\end{proof}

The next lemma discusses the values in the set $T(a,x,b)\cdot x$.

\begin{lemma}
\label{lemma35}
For $T=T(a,x,b)$ we either have $|Tx|\leq 1$ or else $yx\neq zx$ for all distinct $y,z\in T$.
\end{lemma}

\begin{proof}
Suppose that for some $p,q,r,s\in\N$ we have
\begin{equation}
\label{eq36}
a^px=b^r,\ a^{p+q}x=b^{r+s},
\end{equation}
while for some $u,v,w\in\N$ we have
\begin{equation}
\label{eq37}
a^ux=b^w,\ a^{u+v}=b^w.
\end{equation}
From \eqref{eq36}, \eqref{eq37} and Lemma \ref{lemma4} we have:
\begin{alignat}{2}
\label{eq38}
a^{p+qt}x&=b^{r+st} &\quad& (t\in\N),
\\
\label{eq39}
a^{u+vt}x&=b^w && (t\in\N).
\end{alignat}
Let $t_1\in\N$ be such that
\begin{equation}
\label{eq40}
r+st_1>w,
\end{equation}
and let $t_2\in\N$ be such that
\begin{equation}
\label{eq41}
u+vt_2>p+qt_1.
\end{equation}
The inequalities \eqref{eq40}, \eqref{eq41} and relations \eqref{eq38}, \eqref{eq39} with $t=t_1$ and $t=t_2$ respectively contradict Lemma \ref{lemma3}.
\end{proof}

The rest of this section will be devoted to proving that there are only finitely many distinct sets $T(a,x,b)$,
a fact that will be crucial in Section \ref{sec4}.
We accomplish this (in Lemma \ref{lemma10}) by proving that there only finitely many distinct numbers
$q(a,x,b)$ (Lemma \ref{lemma6}), finitely many distinct numbers $p(a,x,b)$ (Lemma \ref{lemma7}), 
and finitely many distinct sets $F(a,x,b)$ (Lemma \ref{lemma9}). We begin, however, with an elementary observation, which must be well known, but we prove it for completeness:

\begin{lemma}
\label{lemma5a}
For every $n\in\N$ and every $r\in\Q^+$ the set
\[
\bigl\{ (m_1,\dots,m_n)\in \N^n\st \frac{1}{m_1}+\dots+\frac{1}{m_n}=r\bigr\}
\]
is finite.
\end{lemma}

\begin{proof}
We prove the assertion by induction on $n$, the case $n=1$ being obvious.
Let $n>1$, and assume the assertion is true for $n-1$.
Consider an $n$-tuple $(m_1,\dots,m_n)\in\N^n$ such that
\[
\frac{1}{m_1}+\dots+\frac{1}{m_n}=r.
\]
Without loss of generality assume $m_1\geq\dots\geq m_n$.
Then we must have $1/m_n \geq r/n$, and so $m_n\leq n/r$. Thus there are only finitely many possible values for $m_n$.
For each of them, the remaining $n-1$ numbers satisfy
\[
\frac{1}{m_1}+\dots+\frac{1}{m_{n-1}}=r-\frac{1}{m_n},
\]
and by induction there are only finitely many such $(n-1)$-tuples.
\end{proof}

\begin{lemma}
\label{lemma6}
The set
\[
\{q(a,x,b)\st a,b\in A,\ x\in S\}
\]
is finite.
\end{lemma}

\begin{proof}
Fix $a\in A$, $x\in S$, and notice that at least one of the sets $T(a,x,b)$ ($b\in A$) is infinite
by Lemma \ref{lemma4a}.
Let
\[
m=\lcm \{ q(a,x,b)\st b\in A,\ |T(a,x,b)|=\infty\}.
\]
Recall that all the sets $F(a,x,b)$ are finite, and let $r\in\N$ be such that
\[
r>\max \bigcup_{b\in A} F(a,x,b).
\]
Let $I=\{ a^r,a^{r+1},\dots,a^{r+m-1}\}$, an `interval'  of size $m$.
From Lemma \ref{lemma5} \ref{T2} we have $I\cap F(a,x,b)=\emptyset$ for all $b\in A$,
so Lemma \ref{lemma4a} implies that $I$ is the disjoint union of sets $I\cap P(a,x,b)$ ($b\in A$).
Since for every $b\in A$ with $P(a,x,b)\neq \emptyset$ we have $q(a,x,b)\mid m$, the set $I$ contains precisely
$m/q(a,x,b)$ elements from $P(a,x,b)$. It follows that
\[
\sum_{b\in A} \frac{m}{q(a,x,b)}=m,
\]
and hence
\[
\sum_{b\in A} \frac{1}{q(a,x,b)}=1.
\]
The assertion now follows from Lemma \ref{lemma5a}.
\end{proof}

\begin{lemma}
\label{lemma7}
The set
\[
\{ p(a,x,b)\st a,b\in A,\ x\in S\}
\]
is finite.
\end{lemma}

\begin{proof}
Fix $a,b\in A$, $x\in S$, and for brevity write $p=p(a,x,b)$, $q=q(a,x,b)$ .
Recall that $p$ has been chosen to be the smallest possible with respect to the condition
that 
\begin{equation}
\label{eq4}
p+qt\in T(a,x,b)\  (t\in\N_0).
\end{equation}
Recalling Lemma \ref{lemma6}, let
\[
Q=\max \{ q(c,y,d)\st c,d\in A,\ y\in S\}.
\]
Assume, aiming for contradiction, that
\[
p>2nQ.
\]
Since $Q\geq q$ we have that $p-2nq>0$.
Consider the $n$ pairs
\[
\{ a^{p-(2t-1)q}, a^{p-2tq}\}\ (t=1,\dots,n).
\]
By Lemmas \ref{lemma3}, \ref{lemma4} and minimality of $p$ we cannot have both members of one of these pairs belong to $T(a,x,b)$.
Hence at least one member in each pairs belongs to some $T(a,x,c)$ with $c\neq b$.
By the pigeonhole principle two of these must belong to the same $T(a,x,c)$, say
\[
a^{p-uq},a^{p-vq}\in T(a,x,c)
\]
for some $1\leq v < u \leq 2n$.
Again by Lemmas \ref{lemma3}, \ref{lemma4} we have that
\[
a^{p-uq+(u-v)qt} \in T(a,x,c)
\]
for all $t\in\N_0$.
On the other hand for $t$ sufficiently large (e.g. $t\geq u$) we have
\[
p-uq+(u-v)qt\geq p,
\]
so that 
\[
p-uq+(u-v)qt=p+wq
\]
for some $w\in\N_0$, and so from \eqref{eq4} we have
\[
a^{p-uq+(u-v)qt}\in T(a,x,b)\neq T(a,x,c),
\]
a contradiction. We conclude that $p(a,x,b)\leq 2nQ$ for all $a,b\in A$, $x\in S$, where the right hand side does not depend on $a$, $b$ or $x$.
\end{proof}

In order to prove our final ingredient, that there are only finitely many distinct sets $F(a,x,b)$, we require one more
elementary fact:

\begin{lemma}
\label{lemma8}
Consider a finite collection of arithmetic progressions:
\[
R_i=\{p_i+tq_i\st t\in\N_0\}\ (i=1,\dots,n).
\]
If there exists $p\in\N$ such that 
\[
[p,\infty)\subseteq \bigcup_{i=1}^n R_i,
\]
then
\[
[p^\prime,\infty)\subseteq \bigcup_{i=1}^n R_i,
\]
where $p^\prime=\max\{p_1,\dots,p_n\}$.
\end{lemma}

\begin{proof}
If $p\leq p^\prime$ there is nothing to prove.
Otherwise the assertion follows from the fact that for every $m\geq p^\prime$ and every $i=1,\dots,n$
we have $m\in R_i$ if and only if $m+q_i\in R_i$.
\end{proof}

\begin{lemma}
\label{lemma9}
The set
\[
\{ F(a,x,b)\st a,b\in A,\ x\in S\}
\]
is finite.
\end{lemma}

\begin{proof}
Fix $a\in A$, $x\in S$.
Finitely many arithmetic progressions $P(a,x,b)$ ($b\in A$, $|T(a,x,b)|=\infty$)
eventually cover the block $N_a$ by Lemmas \ref{lemma4a}, \ref{lemma5}.
Hence, by Lemma \ref{lemma8}, they contain all elements $a^t$ with
\[
t\geq M=\max\{ p(a,x,b)\st b\in A\}.
\]
Hence every $F(a,x,b)$ ($b\in A$) is contained in $\{a,\dots,a^{M-1}\}$.
Since the numbers $p(a,x,b)$ are uniformly bounded by Lemma \ref{lemma7} the assertion follows.
\end{proof}

\begin{lemma}
\label{lemma10}
The set
\[
\{ T(a,x,b)\st a,b\in A,\ x\in S\}
\]
is finite.
\end{lemma}

\begin{proof}
This follows from  Lemmas \ref{lemma5}, \ref{lemma6}, \ref{lemma7}, \ref{lemma9}.
\end{proof}

\section{Finite presentability}
\label{sec3}

Let $S$ be a semigroup, and let $A$ be a generating set for $S$.
Denote by $A^+$ the \emph{free semigroup} on $A$; it consists of all words over $A$.
Let $\epsilon$ denote the \emph{empty word}, and let $A^\ast=A^+\cup\{\epsilon\}$.
Since $A$ is a generating set for $S$, the identity mapping on $A$ induces an epimorphism
$\pi : A^+\rightarrow S$.
The kernel $\ker(\pi)$ is a congruence on $S$; if $R\subseteq A^+\times A^\plus$ is a generating set for this congruence we say that $\langle A\pst R\rangle$ is a presentation for $S$.
We say that $S$ \emph{satisfies a relation} $(u,v)\in A^+\times A^+$ if $\pi(u)=\pi(v)$; we write $u=v$ in this case.
Suppose we are given a set $R\subseteq A^+\times A^+$ and two words $u,v\in A^+$.
We say that the relation $u=v$ is a \emph{consequence} of $R$ if there exist words $u\equiv w_1,w_2,\dots, w_{k-1},w_k\equiv v$ 
($k\geq 1$) such that for each $i=1,\dots,k-1$ we can write
$w_i\equiv \alpha_i u_i \beta_i$ and $w_{i+1}\equiv \alpha_i v_i\beta_i$ where
$(u_i,v_i)\in R$ or $(v_i,u_i)\in R$.

It is well known that the following are equivalent:
\begin{enumerate}[label=\textsf{(P\arabic*)}, leftmargin=*,widest=2]
\item
\label{P1}
$\langle A\pst R\rangle$ is a presentation for $S$.
\item
\label{P2}
$S$ satisfies all relations from $R$, and every relation that $S$ satisfies is a consequence of $R$.
\item
\label{P3}
$S$ satisfies all relations from $R$, and
there exists a set $W\subseteq A^+$ such that $\pi$ maps $W$ injectively into $S$, so that for every $u\in A^+$ there exists $w\in W$ such that $u=w$ is a consequence of $R$.
\end{enumerate}
[\ref{P1}$\Leftrightarrow$\ref{P2} is \cite[Proposition 1.4.2]{lallement}.
\ref{P2}$\Rightarrow$\ref{P3} is proved by choosing a single preimage for every $s\in S$, and letting the resulting set be $W$.
\ref{P3}$\Rightarrow$\ref{P2} is obvious.]
The set $W$ in \ref{P3} is referred to as a set of \emph{normal forms} for elements of $S$.

We are now ready to prove the finitely presented part of the Main Theorem.

\begin{thm}
\label{thm21}
Every semigroup which is a disjoint union of finitely many copies of the free monogenic semigroup is finitely presented.
\end{thm}

\begin{proof}
We continue using notation from Section \ref{sec2}.
Thus $S=\bigcupdot_{a\in A} N_a$, and $N_a=\langle a\rangle$.
The set
\[
W=\{ a^k\st a\in A,\ k\in\N\}
\]
is a set of normal forms for $S$.
Hence for any $a,b\in A$ and $k,l\in\mathbb{N}$ there exist unique
$\alpha(a,k,b,l)\in A$ and $\kappa(a,k,b,l)\in\N$ such that
\begin{equation}
\label{eq1a}
a^kb^l=[\alpha(a,k,b,l)]^{\kappa(a,k,b,l)}.
\end{equation}
It is easy to see that generators $A$ and relations \eqref{eq1a} provide an (infinite) presentation for $S$;
for instance, condition \ref{P3} is clearly satisfied.

Now we claim that the (still infinite) presentation with generators $A$ and relations
 \begin{equation}
 \label{eq4a}
a^k b=[\alpha(a,k,b,1)]^{\kappa(a,k,b,1)},\  (a,b\in A,\ k\in \N)
 \end{equation}
 also defines $S$.
Indeed, the above set of relations is  contained in \eqref{eq1a}, and so $S$ satisfies \eqref{eq4a}. 
We now show that a general relation 
from \eqref{eq1a} is a consequence of \eqref{eq4a}. We do this by induction on $l$.
For $l=1$ we actually have a relation from \eqref{eq4a}, and there is nothing to prove.
Assume the assertion holds for some $l$. Then we have
\[
\begin{array}{lll}
a^k b^{l+1} &\equiv a^kb^l b& \\
&=[\alpha(a,k,b,l)]^{\kappa(a,k,b,l)} b &\text {(by induction)}\\
&\multicolumn{2}{l}{= [\alpha(\alpha(a,k,b,l),\kappa(a,k,b,l),b,1)]^{\kappa(\alpha(a,k,b,l),\kappa(a,k,b,l),b,1)}}\\
&&\text {(by \eqref{eq4a})}\\
&\equiv [\alpha(a,k,b,l+1)]^{\kappa(a,k,b,l+1)}.&\text{(by uniqueness of normal forms)}  
\end{array}
\]
 Therefore, every relation \eqref{eq1a} is a consequence of \eqref{eq4a}.
Since \eqref{eq1a} is a presentation for $S$, so is \eqref{eq4a}.

For any $a,b,c\in A$ consider the set $T(a,b,c)$.
Note that for every $a^i\in T(a,b,c)$ there exists a unique $j\in\N$ such that $a^ib=c^j$.
Let $R_{a,b,c}$ be the set of all these relations; clearly $|R_{a,b,c}|=|T(a,b,c)|$.

Next we claim that for any $a,b,c\in A$
there exists a finite set of relations $R_{a,b,c}^\circ\subseteq R_{a,b,c}$ such that all relations in
$R_{a,b,c}$ are consequences of $R_{a,b,c}^\circ$.
Indeed, if $T(a,b,c)$ is finite (i.e. $|T(a,b,c)|\leq 1$) the assertion is obvious. 
So suppose that $T(a,b,c)$ is infinite.
By Lemma \ref{lemma5} we have
\[
T(a,b,c)=F\cup P,
\]
where $P=\{ a^{p+tq}\st t\in\N_0\}$ and $F\subseteq\{a,\dots,a^{p-1}\}$.
Now, if
\begin{equation}
\label{eq22}
a^pb=c^r,\ a^{p+q}b=c^{r+s},
\end{equation}
then by Lemma \ref{lemma4} we have
\begin{equation}
\label{eq23}
a^{p+tq}b=c^{r+ts}\ (t\in\N_0).
\end{equation}
A closer inspection of the proof of Lemma \ref{lemma4} shows that in fact relations \eqref{eq23} are consequences
of \eqref{eq22}, in the technical sense above.
On the other hand, relations \eqref{eq23} are precisely all the relations $a^ib=c^j$ with $a^i\in P$.
There remain finitely many relations with $a^i\in F$, and the claim follows.

To complete the proof of the theorem, note that the set of defining relations \eqref{eq4a} is the union
$\bigcup_{a,b,c\in A} R_{a,b,c}$.
Hence all these relations are consequences of $\bigcup_{a,b,c\in A}R_{a,b,c}^\circ$,
which is a finite set because $A$ and all $R_{a,b,c}^\circ$ are finite.
\end{proof}

\section{Residual finiteness}
\label{sec4}

A semigroup $S$ is said to be \emph{residually finite} if for any two distinct elements $s, t\in S$ there exists a homomorphism $\phi$ from $S$ into a finite semigroup such that $\phi(s)\neq\phi(t)$.
It is well known that the following are equivalent:

\begin{enumerate}[label=\textsf{(RF\arabic*)}, leftmargin=*,widest=2]
\item
\label{RF1}
$S$ is residually finite.
\item
\label{RF2}
There exists a congruence $\rho$ of \emph{finite index} (i.e. with only finitely many equivalence classes)
such that $(s,t)\not\in\rho$.
\item
\label{RF3}
There exists a right congruence $\rho$ of finite index such that $(s,t)\not\in\rho$.
\end{enumerate}
[\ref{RF1}$\Leftrightarrow$\ref{RF2} is an immediate consequence of the connection between homomorphisms and congruences via kernels.
\ref{RF2}$\Rightarrow$\ref{RF3} is trivial.
\ref{RF3}$\Rightarrow$\ref{RF2} follows from the fact that for a right congruence $\rho$ of
finite index, the largest two-sided congruence contained in $\rho$ also has finite index; see
\cite[Theorem 2.4]{ruskuc98}.]

In this section we prove the residual finiteness part of the Main Theorem, i.e. we prove that every semigroup
which is a disjoint union of finitely many copies of the free monogenic semigroup is residually finite.
So, let $S$ be such a semigroup, and let all the notation be as in Section \ref{sec2}.
Define a relation $\rho$ on $S$ as follows:
\[
(x,y)\in\rho \Leftrightarrow  (\forall z\in S^1)(\exists a\in A)(xz,yz\in N_a).
\]
Intuitively, two elements $(x,y)$ of $S$ are $\rho$-related if every pair of translates by the same element of $S^1$ belongs
 to a single block. In particular, if $(x,y)\in\rho$ then $x$ and $y$ are powers of the same generator $a\in A$, i.e.
\begin{equation}
\label{eq4aa}
\rho\subseteq \bigcup_{a\in A} N_a\times N_a.
\end{equation}
The following is obvious from the definition:

\begin{lemma}
\label{lemma4c}
$\rho$ is a right congruence.
\end{lemma}

An alternative description of $\rho$ is provided by:
\begin{equation}
\label{eq31}
(a^i,a^j)\in\rho \Leftrightarrow (\forall x\in S)(\exists b\in A)(a^i,a^j\in T(a,x,b));
\end{equation}
the proof is obvious.
This description enables us to prove:

\begin{lemma}
\label{lemma32}
$\rho$ has finite index.
\end{lemma}

\begin{proof}
From \eqref{eq31} it follows that the $\rho$-class of an element $a^i\in S$ is
\begin{equation}
\label{eq33}
a^i\rho = \bigcap \{ T(a,x,b)\st x\in S,\ b\in A,\ a^ix\in N_b\}.
\end{equation}
By Lemma \ref{lemma10} there are only finitely many distinct sets $T(a,x,b)$.
Hence there are only finitely many intersections \eqref{eq33}, and the assertion follows.
\end{proof}

For each $a\in A$, consider the restriction
\[
\rho_a = \rho |_{N_a}.
\]
From Lemmas \ref{lemma4c}, \ref{lemma32} it follows that $\rho_a$ is a right congruence of finite index on $N_a$.
But $N_a$, being free monogenic, is commutative, and so $\rho_a$ is actually a congruence.
Furthermore, congruences on a free monogenic semigroup are well understood, and we have that
\[
\rho_a=\{ (a^i,a^j)\st i=j \mbox{ or } 
(i,j\geq p_a\ \&\ i\equiv j\pmod{q_a} ) \},
\]
for some $p_a,q_a\in \N$; see \cite[Section 1.2]{howie95}.

Motivated by this, for any pair $(x,y)\in \rho$ we define their \emph{distance} as
\[
d(x,y)=\frac{|i-j|}{q_a} \ \ \mbox{if } x=a^i,\ y=a^j.
\]

\begin{lemma}
\label{lemma41}
If $x,y,z\in S$ are such that $(x,y)\in\rho$ and $x\neq y$ then
\[
d(x,y) \mid d(xz,yz).
\]
\end{lemma}

\begin{proof}
Since $\rho$ is a right congruence we have $(xz,yz)\in\rho$, and so 
$d(xz,yz)$ is defined.
If $xz=yz$ there is nothing to prove, so suppose $xz\neq yz$.
Write
\[
x=a^r,\ y=a^s,
\]
where $r,s\geq p_a$, $r\equiv s\pmod{q_a}$, $r\neq s$.
Without loss of generality assume $s>r$ so that $s=r+tq_a$ for some $t\in\N$.
Notice that $(a^r,a^{r+q_a})\in \rho$; furthermore we must have
$a^rz\neq a^{r+q_a}z$ by Lemma \ref{lemma35}.
Therefore
\begin{equation}
\label{eq42}
a^rz=b^u,\ a^{r+q_a}z=b^v,
\end{equation}
for some $u,v\geq p_b$, $u\equiv v\pmod{q_b}$, $u<v$.
Write $v=u+wq_b$, $w\in\N$.
Equalities \eqref{eq42} become
\[
a^rz=b^u,\ a^{r+q_a}z=b^{u+wq_b},
\]
and Lemma \ref{lemma4} yields
\[
a^{r+tq_a}z=b^{u+twq_b}.
\]
Therefore
\[
d(xz,yz) = d(a^rz,a^{r+tq_a}z)=d(b^u,b^{u+twq_b})
= tw=w d(a^r,a^{r+tq_a})=w d(x,y),
\]
as required.
\end{proof}

We are now ready to prove:

\begin{thm}
\label{thm61}
Every semigroup which is a disjoint union of finitely many copies of the free monogenic semigroup is residually finite.
\end{thm}

\begin{proof}
Let $S$ be such a semigroup, with all the foregoing notation remaining in force.
Let $x,y\in S$ be two arbitrary distinct elements.
By \ref{RF3} it is sufficient to prove that $x$ and $y$ are separated by a right congruence of finite index.
If $(x,y)\not\in\rho$ then $\rho$ is such a congruence by Lemmas \ref{lemma4c}, \ref{lemma32}.
So suppose $(x,y)\in\rho$, say with $x,y\in N_b$, and let
\[
d(x,y)=d>0.
\]
Let $\sigma$ be the right congruence on $S$ generated by the set
\[
G=\{ (a^{p_a},a^{p_a+2dq_a}) \st a\in A\}.
\]
Clearly $\sigma$ is a refinement of $\rho$ (i.e. $\sigma\subseteq\rho$).
Notice that $G$ contains one pair of distinct elements from each block $N_a$.
Hence the restriction of $\sigma$ to each $N_a$ is a non-trivial congruence, and so has finite index.
Therefore $\sigma$ itself has finite index too.

We claim that $(x,y)\not\in\sigma$.
Suppose otherwise; this means that there is a sequence
\[
x=u_1,u_2,\dots,u_m=y
\]
of elements of $S$, such that for each $i=1,\dots,m-1$ we can write
\[
u_i=v_iz_i,\ u_{i+1}=w_iz_i,
\]
for some $v_i,w_i\in S$, $z_i\in S^1$, satisfying $(v_i,w_i)\in G$ or $(w_i,v_i)\in G$.
(This is a well known general fact; see for example \cite[Section 8.1]{howie95} .)
Without loss of generality we may assume that all $u_i$ are distinct.
From $\sigma\subseteq\rho$ it follows that all $u_i$ belong to the block $N_b$,
say
\[
u_i = b^{s_i} \ (i=1,\dots,m).
\]
By definition of $G$ and Lemma \ref{lemma41} we have that 
$2d\mid d(u_i,u_{i+1})$ for all $i=1,\dots,m-1$.
This is equivalent to
\[
s_i\equiv s_{i+1} \pmod{2q_b d}\ (i=1,\dots,m-1),
\]
from which it follows that $s_1\equiv s_m\pmod{2q_bd}$,
and hence $2d\mid d(x,y)=d$, a contradiction.
\end{proof}

\section{Concluding remarks}
\label{sec5}

Arguably, the free monogenic semigroup $N$ is the most fundamental commutative semigroup.
It is well known that all finitely generated commutative semigroups are finitely presented 
and
residually finite. 
Finite presentability was first proved by R\'{e}dei \cite{redei65}; see also \cite[Section 9]{grillet95}.
Residual finiteness was proved by Malcev \cite{malcev58}; see also \cite{carlisle71,lallement71}.
In this paper we have shown that disjoint unions of copies of $N$ 
(which, of course, need not be commutative) in this respect behave like commutative semigroups.
It would be interesting to know if this generalises to unions of commutative semigroups:

\begin{question}
Is it true that every semigroup which is a finite disjoint union of finitely generated commutative semigroups is
necessarily: (a) finitely presented; (b) residually finite?
\end{question}

By way of contrast, there is no reason to believe that our results would generalise to disjoint unions of copies of a free (non-commutative) semigroup of rank $>1$.

\begin{question}
Does there exist a semigroup $S$ which is a disjoint union of two copies of a free semigroup of rank $2$ which is not: (a) finitely presented; (b) residually finite?
\end{question}

Finally, it would be interesting to know how the subsemigroups of semigroups investigated in this paper behave.
Since residual finiteness is preserved under taking substructures, they are certainly all residually finite.
Also, they are all finitely generated, which follows from the observation that all subsemigroups of $N$ are finitely generated.

\begin{question}
Is every subsemigroup of every semigroup which is a disjoint union of finitely many copies of the free monogenic semigroup finitely presented?
\end{question}

\begin{flushleft}
School of Mathematics and Statistics\\
University of St Andrews\\
St Andrews KY16 9SS\\
Scotland, U.K.\\
\smallskip
\texttt{\{nabilah,nik\}@mcs.st-and.ac.uk}
\end{flushleft}

\end{document}